\documentclass{biometrika}
\usepackage{amsmath}
\usepackage{amssymb}
\usepackage{mathtools}
\usepackage{bm}
\usepackage{microtype}
\usepackage[dvipsnames]{xcolor}
\usepackage[colorlinks=true,citecolor=MidnightBlue,linkcolor=MidnightBlue,urlcolor=MidnightBlue]{hyperref}
\hypersetup{
  pdftitle={Tensor-normal maximum likelihood estimation at the operator-norm sample threshold},
  pdfauthor={Hengzhi He and Guang Cheng},
  pdfsubject={Maximum likelihood estimation in tensor-normal models},
  pdfkeywords={tensor normal model, Kronecker covariance, maximum likelihood, operator norm, geodesic convexity}
}
\newcommand{\R}{\mathbb{R}}
\newcommand{\cP}{\mathcal{P}}
\newcommand{\cG}{\mathcal{G}}
\newcommand{\cE}{\mathcal{E}}
\newcommand{\Mat}{\operatorname{Mat}}
\newcommand{\Sym}{\operatorname{Sym}}
\newcommand{\PD}{\operatorname{PD}}
\newcommand{\tr}{\operatorname{tr}}

\newcommand{\eps}{\varepsilon}
\newcommand{\dmax}{d_{\max}}
\newcommand{\dmin}{d_{\min}}
\newcommand{\norm}[1]{\left\lVert #1\right\rVert}
\newcommand{\abs}[1]{\left\lvert #1\right\rvert}
\newcommand{\ip}[2]{\left\langle #1,#2\right\rangle}
\newcommand{\Prob}{\mathbb{P}}
\newcommand{\E}{\mathbb{E}}
\newcommand{\PiZero}{\Pi_0}
\numberwithin{equation}{section}
\makeatletter
\def\ps@plain{\ps@empty}
\let\ps@biom\ps@empty
\let\ps@biometrika\ps@empty
\makeatother
\begin{document}
\jname{}
\jyear{}
\copyrightinfo{}
\title{Tensor-normal maximum likelihood estimation at the operator-norm sample threshold}
\author{Hengzhi He and Guang Cheng}
\affil{Department of Statistics and Data Science, University of California, Los Angeles\
\email{hengzhihe@g.ucla.edu; guangcheng@stat.ucla.edu}}
\markboth{H. He and G. Cheng}{Tensor-normal maximum likelihood estimation}
\maketitle
\begin{abstract}
Let $X_1,\ldots,X_n$ be independent Gaussian tensors in
$\R^{d_1}\otimes\cdots\otimes\R^{d_k}$ whose covariance is a Kronecker
product of $k$ unknown positive-definite factors, and put
$D=\prod_{a=1}^k d_a$ and $\dmax=\max_a d_a$.  A recent result of
\citet{franks2026near} established condition-number-free
guarantees for the tensor-normal maximum likelihood estimator under the
sample threshold $nD\gtrsim k^2\dmax^3$ and asked whether the cubic dependence
on $\dmax$ could be replaced by $\dmax^2$.  We prove that it can.  For
$t\geq1$, if $nD\geq Ck^2\dmax^2t^2$, then with high probability the maximum
likelihood estimator exists uniquely and
\[
 d_{\rm FR}(\widehat\Theta,\Theta)
 \leq C t\sqrt{k}\,\dmax/\sqrt n,
 \qquad
 d_{\rm FR}(\widehat\Theta_a,\Theta_a)
 \leq C t\sqrt{k d_a}\,\dmax/\sqrt{nD}.
\]
For every mode of largest dimension, we also obtain the sharp Thompson bound
$d_{\rm op}(\widehat\Theta_a,\Theta_a)\leq
Ct\dmax/\sqrt{nD}$.
The bounds are uniform over the unknown factors and require no condition-number
bound or sparsity.  Gaussian submodel lower bounds match the full and
largest-factor Fisher--Rao rates up to $\sqrt{k}$ and the largest-factor
Thompson rate up to constants.  For fixed $k$, the dependence of the sample
threshold on $\dmax$ is optimal.

GPT-5.6 Sol and Claude Fable 5 were used to assist with proof development, verification and paper writing. 
\end{abstract}
\section{Introduction}
Consider the tensor-normal model
\begin{equation}
 X_i\sim N\left(0,\Sigma_1\otimes\cdots\otimes\Sigma_k\right),
 \qquad i=1,\ldots,n,                                      \label{eq:model-intro}
\end{equation}
where $\Sigma_a\in\PD(d_a)$, $D=\prod_a d_a$,
$\dmax=\max_a d_a$ and $\dmin=\min_a d_a$.  The likelihood may be unbounded,
and its maximizer need not exist or be unique.  Exact almost-sure existence thresholds have been
studied algebraically \citep{drton2021existence,derksen2021matrix,
derksen2022tensor}.  We instead study condition-number-free nonasymptotic
localization of the maximum likelihood estimator.

The direct predecessor is \citet{franks2026near}.  They proved existence,
uniqueness and Fisher--Rao error bounds under
$nD\gtrsim k^2\dmax^3$, and asked whether the cubic dependence on $\dmax$
could be replaced by $\dmax^2$.  Theorem~\ref{thm:main} answers this question
affirmatively: for $t\geq1$, the condition
$nD\geq Ck^2\dmax^2t^2$ suffices.  For $k=2$ this also gives the threshold
$n\gtrsim\dmax/\dmin$ without logarithmic factors, although our bound for the
smaller factor is weaker than the refined rate in their Theorem 1.11.

The proof starts from Theorem C.1 in the supplementary material to
\citet{franks2026near}, which controls the entire traceless matrix space.  It
therefore gives a random Gram bound on
\[
 \R\oplus\bigoplus_{a=1}^k\Mat_0(d_a;\R).
\]
A tangent transported away from the identity is generally nonsymmetric, so
this full-space statement is essential.  Exact conjugation transports the
Gram bound to a fixed Thompson ball.  We then combine sensitivity of the
constrained optimizer, an equivariant Kirszbraun extension and Gaussian
concentration to localize each factor in operator norm.  The resulting
optimizer lies in the interior of the ball and is the unconstrained MLE.

The paper is organized as follows.  Section 2 states the main result.
Sections 3--6 prove the Gram, transport, sensitivity and localization lemmas.
Section 7 proves the theorem, and Section 8 gives matching Gaussian lower
bounds.
\section{Model, geometry and main result}
\subsection{Likelihood and canonical factors}
Throughout, $\PD(d)$ denotes the cone of real $d\times d$ positive-definite
matrices,
\[
 \Sym_0(d)=\{A\in\R^{d\times d}:A=A^T,\ \tr A=0\},
 \qquad
 \Mat_0(d;\R)=\{A\in\R^{d\times d}:\tr A=0\}.
\]
It is convenient to parameterize by the precision factors
$\Theta_a=\Sigma_a^{-1}$.  Let
\[
 \cP=\left\{\Theta_1\otimes\cdots\otimes\Theta_k:
             \Theta_a\in\PD(d_a)\right\}.
\]
The factors are identifiable only up to reciprocal scalings.  We use the
canonical convention that $(\det\Theta_a)^{1/d_a}$ is the same for every
$a$.  Set
\[
 D=\prod_{a=1}^k d_a,\qquad \dmin=\min_a d_a,\qquad
 \dmax=\max_a d_a,\qquad M=nD,\qquad
 \rho_x=\frac{1}{M}\sum_{i=1}^n x_ix_i^T.
\]
Up to an irrelevant additive and multiplicative constant, the negative
log-likelihood is
\begin{equation}
 f_x(\Theta)=\tr(\rho_x\Theta)-{1\over D}\log\det\Theta,
 \qquad \Theta\in\cP.                                      \label{eq:likelihood}
\end{equation}
Whenever the minimizer exists uniquely, it is denoted by
$\widehat\Theta=\bigotimes_a\widehat\Theta_a$.
For an operator $A$ on $\bigotimes_a\R^{d_a}$ and a set of modes $J$, let
$A^{(J)}$ be its partial trace, characterized by
\begin{equation}
 \tr\{A^{(J)}H\}=\tr\{A H^{(J)}\},                         \label{eq:partial-trace}
\end{equation}
where $H^{(J)}$ acts as $H$ on the modes in $J$ and as the identity on the
remaining modes.  We abbreviate $A^{(a)}$ and $A^{(ab)}$.
\subsection{Affine-invariant coordinates}
At the identity, write a tangent as
\[
 H=(h_0,H_1,\ldots,H_k),\qquad
 H_a\in\Sym_0(d_a),\qquad
 \norm H^2=h_0^2+\sum_a\norm{H_a}_F^2,
\]
and associate the ambient generator
\begin{equation}
 K_H=h_0I_D+\sum_{a=1}^k\sqrt{d_a}\,H_a^{(a)}.                \label{eq:generator}
\end{equation}
If $\Theta=\bigotimes_a\Theta_a$, the corresponding exponential map is
\begin{align}
 \operatorname{Exp}_{\Theta}(H)
 &=\Theta^{1/2}e^{K_H}\Theta^{1/2} \notag\\
 &=e^{h_0}\bigotimes_{a=1}^k
   \left(\Theta_a^{1/2}e^{\sqrt{d_a}H_a}\Theta_a^{1/2}\right). \label{eq:exponential-map}
\end{align}
Here and below the square root is the principal square root.  The normalized
affine-invariant distance $d$ on $\cP$ is chosen so that a constant-speed
geodesic generated by $H$ has speed $\norm H$.  For positive-definite
matrices $A,B$ of the same dimension, the Fisher--Rao distance of the
centered Gaussian model is
\begin{equation}
 d_{\rm FR}(A,B)
 ={1\over\sqrt2}\norm{\log(A^{-1/2}BA^{-1/2})}_F.             \label{eq:fisher-rao-definition}
\end{equation}
For full $D\times D$ precisions, our normalized distance is therefore
\begin{equation}
 d(\Theta,\Theta')=\sqrt{2\over D}\,
 d_{\rm FR}(\Theta,\Theta').                                  \label{eq:metric-normalization}
\end{equation}
In particular, at the identity, if
$\theta=e^{h_0}\bigotimes_aP_a$ with $\det P_a=1$ and
$F_a=\log P_a$, then
\begin{equation}
 H_a={F_a\over\sqrt{d_a}},
 \qquad
 d(I_D,\theta)^2=h_0^2+\sum_{a=1}^k{\norm{F_a}_F^2\over d_a}. \label{eq:coordinate-distance-prelim}
\end{equation}
For positive-definite matrices of any dimension, we use the Thompson
distance
\[
 d_{\rm op}(A,B)=
 \norm{\log(A^{-1/2}BA^{-1/2})}_{\rm op}.
\]
The function \eqref{eq:likelihood} is geodesically convex on $\cP$;
see \citet{wiesel2012geodesic} and \citet[\S2.1]{franks2026near}.
For later use, if $x$ consists of standard Gaussian tensors, its gradient
at the identity is
\begin{equation}
 \nabla_0f_x=\tr\rho_x-1,
 \qquad
 \nabla_af_x=\sqrt{d_a}
 \left(\rho_x^{(a)}-{\tr\rho_x\over d_a}I_{d_a}\right).       \label{eq:gradient}
\end{equation}
We use the following concentration result from Proposition 2.11 of
\citet{franks2026near}: if $0<\eps<1$ and
$M\geq\dmax^2/\eps^2$, then, outside an event of probability at most
$2(k+1)\exp\{-\eps^2M/(8\dmax)\}$,
\begin{equation}
 \abs{\nabla_0f_x}\leq\eps,
 \qquad
 \norm{\nabla_af_x}_{\rm op}\leq{9\eps\over\sqrt{d_a}},
 \qquad
 \norm{\nabla f_x}\leq\sqrt{82k}\,\eps.                    \label{eq:gradient-concentration}
\end{equation}
\subsection{Main theorem}
Modes with $d_a=1$ have no traceless parameter and may be removed.  We state
the theorem for $d_a\geq2$; $k$ can always be replaced by the number of
nontrivial modes.
\begin{theorem}[Tensor-normal MLE at the quadratic threshold]
\label{thm:main}
There are universal constants $C,c,c'>0$ such that the following holds.
Let $t\geq1$, let $X_1,\ldots,X_n$ follow \eqref{eq:model-intro}, and suppose
\begin{equation}
 M=nD\geq Ck^2\dmax^2t^2.                                    \label{eq:main-threshold}
\end{equation}
For $a<b$, define
\begin{equation}
 p_{ab}=\min\left\{1,
 \left({c\sqrt M\over k(d_a+d_b)}\right)^{-c'(d_a+d_b)/2}
 \right\}.                                                    \label{eq:pab}
\end{equation}
Then, with probability at least
\begin{equation}
 1-Ck\exp\left\{-{cM\over k\dmax}\right\}
   -\sum_{a<b}p_{ab}-Ck\exp\{-ct^2\dmax\},                   \label{eq:success-probability}
\end{equation}
the MLE exists uniquely and
\begin{align}
 d_{\rm FR}(\widehat\Theta,\Theta)
 &\leq C{\sqrt{k}\,\dmax\over\sqrt n}\,t,                  \label{eq:main-full-rate}\\
 d_{\rm FR}(\widehat\Theta_a,\Theta_a)
 &\leq C{\sqrt{k d_a}\,\dmax\over\sqrt{nD}}\,t,
 \qquad a=1,\ldots,k,                                        \label{eq:main-factor-rate}\\
 \max_{a:d_a=\dmax}d_{\rm op}(\widehat\Theta_a,\Theta_a)
 &\leq C{\dmax\over\sqrt{nD}}\,t.                           \label{eq:main-largest-op-rate}
\end{align}
The factor representatives in these bounds are the canonical representatives
specified in Section~2.1.
In particular, a coarser form of the exceptional probability is
\begin{equation}
 Ck e^{-ct^2\dmax}
 +Ck^2\left({c\sqrt M\over k\dmax}\right)^{-c'\dmin},        \label{eq:coarse-probability}
\end{equation}
after changing universal constants.
\end{theorem}
The same argument gives a dimension-adaptive bound for every mode.  We record
it separately because simultaneous control of small modes has a different
tail probability.
\begin{corollary}[Factorwise Thompson bounds]
\label{cor:factor-op}
Under the assumptions of Theorem~\ref{thm:main}, with probability at least
\begin{equation}
 1-Ck\exp\left\{-{cM\over k\dmax}\right\}
 -\sum_{a<b}p_{ab}-C\sum_{a=1}^k e^{-ct^2d_a},                \label{eq:op-success-probability}
\end{equation}
the MLE exists uniquely and, simultaneously for every $a$,
\begin{equation}
 d_{\rm op}(\widehat\Theta_a,\Theta_a)
 \leq {Ct\over\sqrt{nD}}
       \left(d_a+{\dmax\over\sqrt{k}}\right).                \label{eq:factor-op-rate}
\end{equation}
For the largest modes, the additional failure probability is
$O(ke^{-ct^2\dmax})$ and is absorbed by \eqref{eq:success-probability}, giving
\eqref{eq:main-largest-op-rate}.
\end{corollary}
For modes of small dimension the tail $e^{-ct^2d_a}$ in
\eqref{eq:op-success-probability} is weak: simultaneous control of every
mode at confidence $1-\delta$ requires
$t\gtrsim\{\log(k/\delta)/\dmin\}^{1/2}$.
The sharper pairwise probability \eqref{eq:success-probability} is the
appropriate statement if $k$ is allowed to grow.  For fixed $k$, the
condition is simply $n\gtrsim\dmax^2/D$ at constant confidence, subject of
course to the integer constraint $n\geq1$.
\section{A full-Lie orbit Gram bound}
\label{sec:full-lie}
The central observation is that the random channel estimate used in the
earlier work is naturally a statement on all traceless matrices.  Define
\[
 \mathfrak g=\R\oplus\bigoplus_{a=1}^k\Mat_0(d_a;\R),
 \qquad
 \norm Z^2=z_0^2+\sum_a\norm{Z_a}_F^2,
\]
where the $Z_a$ need not be symmetric, and put
\begin{equation}
 K_Z=z_0I_D+\sum_{a=1}^k\sqrt{d_a}\,Z_a^{(a)},
 \qquad
 \cE_x(Z)={1\over M}\sum_{i=1}^n\norm{K_Zx_i}_2^2.            \label{eq:orbit-energy}
\end{equation}
\begin{lemma}[Full-Lie Gram event]
\label{lem:full-lie}
Let $x_1,\ldots,x_n$ be independent $N(0,I_D)$ tensors.  If
$M\geq Ck^2\dmax^2$, there is an event $\cG$, invariant under
$\prod_aO(d_a)$, on which
\begin{equation}
 \lambda_0\norm Z^2\leq\cE_x(Z)\leq L_0\norm Z^2,
 \qquad Z\in\mathfrak g,                                     \label{eq:full-lie-bound}
\end{equation}
for universal constants $0<\lambda_0<L_0<\infty$.  Moreover,
\begin{equation}
 \Prob(\cG^c)\leq
 Ck\exp\left\{-{cM\over k\dmax}\right\}+\sum_{a<b}p_{ab},  \label{eq:gram-failure}
\end{equation}
where $p_{ab}$ is defined in \eqref{eq:pab}.
Here $(U_1,\ldots,U_k)\in\prod_aO(d_a)$ acts on the data by
$x_i\mapsto(\bigotimes_aU_a)x_i$.
\end{lemma}
\begin{proof}
Write $q=\tr\rho_x$ and
\[
 \Delta_a=d_a\rho_x^{(a)}-qI_{d_a}.
\]
Expanding \eqref{eq:orbit-energy} gives
\begin{align}
 \cE_x(Z)
 &=qz_0^2
 +2z_0\sum_a\sqrt{d_a}\,\tr(\rho_x^{(a)}Z_a)
 +\sum_a d_a\tr(\rho_x^{(a)}Z_a^TZ_a) \notag\\
 &\quad+2\sum_{a<b}\sqrt{d_ad_b}\,
       \tr\{\rho_x^{(ab)}(Z_a^T\otimes Z_b)\}.               \label{eq:energy-expansion}
\end{align}
The diagonal blocks obey
\begin{equation}
 \left|d_a\tr(\rho_x^{(a)}Z_a^TZ_a)-q\norm{Z_a}_F^2\right|
 \leq\norm{\Delta_a}_{\rm op}\norm{Z_a}_F^2.                 \label{eq:diagonal-error}
\end{equation}
The scalar-to-mode block has norm
$\norm{\Delta_a}_F/\sqrt{d_a}$.
For a pair $a<b$, flatten the samples over these two modes.  This produces
$N=M/(d_ad_b)$ independent $d_a\times d_b$ standard Gaussian Kraus
matrices $A_1,\ldots,A_N$ and the map
\begin{equation}
 \Phi_{ab}(K)={1\over M}\sum_{r=1}^N A_rKA_r^T.               \label{eq:pair-map}
\end{equation}
With the Frobenius inner product, the cross term in
\eqref{eq:energy-expansion} is
\begin{equation}
 \tr\{\rho_x^{(ab)}(Z_a^T\otimes Z_b)\}
 =\ip{Z_a}{\Phi_{ab}(Z_b^T)}_F.                               \label{eq:cross-map}
\end{equation}
Let $\Pi_a$ be orthogonal projection onto $\Mat_0(d_a)$ and let
$T(K)=K^T$.  Thus the cross block on traceless spaces is
\[
 B_{ab}=\sqrt{d_ad_b}\,\Pi_a\Phi_{ab}T\Pi_b.
\]
The transpose is an isometry and commutes with traceless projection.
We use the following form of the random-channel bound of
\citet{franks2026near}, stated as Theorem C.1 in their supplementary
material and invoked there, with Kraus dimensions $(d_a,d_b)$, in the proof
of their Theorem 2.16: for every $s\geq2$, with failure probability at most
$s^{-c(d_a+d_b)}$,
\begin{equation}
 \norm{\left(\sum_{r=1}^N A_r\otimes A_r\right)\Pi_b}_{F\to F}
 \leq Cs^2\sqrt N(d_a+d_b),                                  \label{eq:pisier-input}
\end{equation}
where $\norm{\cdot}_{F\to F}$ denotes the operator norm induced by the
Frobenius norm on matrix space.  Importantly, \eqref{eq:pisier-input} is a
bound on the entire traceless matrix space, not only on its symmetric part:
it controls the composition $\Phi\circ\Pi$ on all of $\Mat(d_b;\R)$,
whereas the quantum-expansion functional of their Definition 2.14 restricts
both arguments to traceless symmetric matrices.  Set
\begin{equation}
 s_{ab}^2={c_0\sqrt M\over k(d_a+d_b)}.                         \label{eq:sab}
\end{equation}
The constants are chosen in the following order: the small numerical
$\beta$ below and a target $c_1$ are fixed first, then $c_0$ is taken small
enough that \eqref{eq:cross-bound} holds with this $c_1$, and finally the
sample constant $C$ in the hypothesis $M\geq Ck^2\dmax^2$ is taken large;
since $d_a+d_b\leq2\dmax$, this last choice makes $s_{ab}\geq2$.  Since
$N=M/(d_ad_b)$, equations \eqref{eq:pair-map} and
\eqref{eq:pisier-input} yield
\begin{equation}
 \norm{B_{ab}}_{F\to F}
 \leq {\sqrt{d_ad_b}\over M}
 Cs_{ab}^2\sqrt N(d_a+d_b)\leq{c_1\over k},                  \label{eq:cross-bound}
\end{equation}
where $c_1$ is made arbitrarily small by choosing $c_0$ small.  The failure
probability is bounded by \eqref{eq:pab}, after adjusting constants.
Fix a small numerical $\beta>0$ and apply
\eqref{eq:gradient-concentration} with
$\eps=\beta/(9\sqrt{k})$.  Outside an event of probability at most
$Ck\exp\{-cM/(k\dmax)\}$,
\begin{equation}
 |q-1|\leq{\beta\over9\sqrt{k}},
 \qquad
 \max_a\norm{\Delta_a}_{\rm op}\leq{\beta\over\sqrt{k}}.  \label{eq:marginal-good}
\end{equation}
The prerequisite $M\geq\dmax^2/\eps^2$ in
\eqref{eq:gradient-concentration} is satisfied because it is
$M\geq81k\dmax^2/\beta^2$, which follows from
$M\geq Ck^2\dmax^2$ once the universal constant $C$ is large enough.
The block matrix formed by the mode-to-mode off-diagonal terms has norm at
most $(k-1)c_1/k$.  Indeed, for a symmetric block matrix $B$,
\begin{equation}
 \norm{B}_{\rm op}
 \leq\norm{(\norm{B_{ab}}_{\rm op})_{a,b}}_{\rm op}
 \leq\max_a\sum_{b\ne a}\norm{B_{ab}}_{\rm op},              \label{eq:block-gershgorin}
\end{equation}
which is the block Gershgorin estimate used here.  The diagonal perturbation
is at most $\beta/\sqrt{k}$, while the norm of the scalar-to-mode star is
bounded by
\begin{equation}
 \left(\sum_a{\norm{\Delta_a}_F^2\over d_a}\right)^{1/2}
 \leq\left(\sum_a\norm{\Delta_a}_{\rm op}^2\right)^{1/2}
 \leq\beta.                                                    \label{eq:star-bound}
\end{equation}
Therefore the full Gram operator differs from $qI$ by less than a fixed
constant, which proves \eqref{eq:full-lie-bound} after choosing $c_1$ and
$\beta$ sufficiently small.  The union bound over the pairs gives
\eqref{eq:gram-failure}.
The event can be defined as the event that the finite-dimensional Gram
operator has spectrum in $[\lambda_0,L_0]$, so its invariance under the
local orthogonal group is immediate.
\end{proof}
\section{Transport to a fixed Thompson ball}
\label{sec:transport}
For a fixed constant $R>0$, define
\begin{equation}
 \mathcal C_R=\left\{\Theta\in\cP:
 e^{-R}I_D\preceq\Theta\preceq e^RI_D\right\}.                 \label{eq:thompson-ball}
\end{equation}
\medskip\noindent{\itshape Choice of the radius.}
We henceforth fix $R=1$; all constants depending on $R$ are universal.
The set $\mathcal C_R$ is closed, compact and geodesically convex.  Indeed,
write its canonical factors as
$\Theta=e^{h_0}\bigotimes_aP_a$, with $\det P_a=1$.  The order interval
gives $|h_0|\leq R$ and
$\kappa(P_a)\leq\kappa(\Theta)\leq e^{2R}$.  Determinant normalization then
bounds every eigenvalue of every $P_a$ above and away from zero.  Every
sequence in $\mathcal C_R$ therefore has a subsequence whose canonical factors
converge to a point still in $\mathcal C_R$.  Thus $\mathcal C_R$ is compact
and hence closed.
The Kronecker manifold is totally geodesic under the affine-invariant
metric, and monotonicity of the weighted matrix geometric mean preserves
the order interval in \eqref{eq:thompson-ball}; see
\citet{wiesel2012geodesic}.
\begin{lemma}[Uniform Hessian bounds]
\label{lem:hessian-transport}
On the event $\cG$ of Lemma~\ref{lem:full-lie},
\begin{equation}
 \lambda_R I\preceq\operatorname{Hess}f_x(\Theta)
 \preceq L_RI,
 \qquad \Theta\in\mathcal C_R,                                \label{eq:uniform-hessian}
\end{equation}
where $\lambda_R=\lambda_0e^{-3R}$ and $L_R=L_0e^{3R}$.
\end{lemma}
\begin{proof}
Let $G=\Theta^{1/2}$ be the principal product square root.  For a symmetric
statistical tangent $H=(h_0,H_1,\ldots,H_k)$, the Hessian identity and its
equivariance \citep[Lemma 2.12 and Remark 2.13]{franks2026near} give
\begin{equation}
 \operatorname{Hess}f_x(\Theta)[H,H]
 ={1\over M}\sum_{i=1}^n\norm{K_HGx_i}_2^2.                    \label{eq:hessian-away}
\end{equation}
At the level of the product representation,
\begin{equation}
 K_HG=GK_Z,
 \qquad z_0=h_0,\qquad Z_a=G_a^{-1}H_aG_a.                     \label{eq:intertwining}
\end{equation}
Indeed, mode by mode,
$H_a^{(a)}G=G(G_a^{-1}H_aG_a)^{(a)}$; summing these identities and the scalar
term gives \eqref{eq:intertwining}.
Although $H_a$ is symmetric, $Z_a$ need not be.  This is the reason the
Gram event in Lemma~\ref{lem:full-lie} must hold on all of
$\Mat_0(d_a)$.
The order bound in \eqref{eq:thompson-ball} implies
\begin{equation}
 \sigma(G)^2\subset[e^{-R},e^R],
 \qquad
 \kappa(G_a)\leq\kappa(G)=\prod_b\kappa(G_b)\leq e^R.          \label{eq:condition-factors}
\end{equation}
Consequently,
\begin{equation}
 e^{-2R}\norm H^2\leq\norm Z^2\leq e^{2R}\norm H^2.           \label{eq:conjugation-norm}
\end{equation}
Combining \eqref{eq:hessian-away}, \eqref{eq:intertwining},
Lemma~\ref{lem:full-lie}, the singular-value bound on $G$, and
\eqref{eq:conjugation-norm} proves \eqref{eq:uniform-hessian}.  The three
exponential factors correspond respectively to multiplication by $G$ and
the two-sided conjugation of each tangent block.
\end{proof}
\section{Sensitivity of a constrained maximum likelihood estimator}
\label{sec:sensitivity}
For data $x$ in $\cG$, let $\theta_x$ denote the minimizer of $f_x$ over
$\mathcal C_R$.  It exists by compactness and is unique by
Lemma~\ref{lem:hessian-transport}.
\begin{lemma}[Lipschitz solution map]
\label{lem:solution-map}
For $x,y\in\cG$,
\begin{equation}
 d(\theta_x,\theta_y)
 \leq {C_R\over\sqrt M}\norm{x-y}_2,                           \label{eq:solution-lipschitz}
\end{equation}
where $x$ and $y$ denote the corresponding stacked samples in $\R^M$.
\end{lemma}
\begin{proof}
Let $d_*=d(\theta_x,\theta_y)$.  Strong geodesic convexity and the
variational inequality for a constrained minimizer give
\begin{align}
 f_x(\theta_y)-f_x(\theta_x)&\geq {\lambda_R\over2}d_*^2,
 \notag\\
 f_y(\theta_x)-f_y(\theta_y)&\geq {\lambda_R\over2}d_*^2.
                                                                    \label{eq:two-kkt}
\end{align}
To see the constrained first-order step explicitly, the directional
derivative at a minimizer is nonnegative along every feasible geodesic in
$\mathcal C_R$.  Integrating the lower Hessian bound along the geodesic from
that minimizer to the other point gives each inequality in
\eqref{eq:two-kkt}.
Adding these inequalities yields
\begin{equation}
 \lambda_Rd_*^2
 \leq(f_x-f_y)(\theta_y)-(f_x-f_y)(\theta_x).                  \label{eq:kkt-sum}
\end{equation}
Let $\Theta_s$, $0\leq s\leq1$, be the constant-speed geodesic between
the two minimizers and define
\[
 A_s=\Theta_s^{-1/2}\dot\Theta_s\Theta_s^{-1/2}.
\]
If $H_s$ is the corresponding statistical tangent, then $A_s=K_{H_s}$ and
the normalization in \eqref{eq:metric-normalization} gives
\begin{equation}
 {1\over D}\tr(A_s^2)=\norm{H_s}^2=d_*^2.                     \label{eq:speed-normalization}
\end{equation}
Because the log-determinant terms cancel from $f_x-f_y$, differentiation
along the geodesic and the identity
$xx^T-yy^T=(x-y)x^T+y(x-y)^T$ give
\begin{align}
 \left|{d\over ds}(f_x-f_y)(\Theta_s)\right|
 &\leq {1\over M}\norm{\Theta_s^{1/2}(x-y)}_2
 \left\{\norm{A_s\Theta_s^{1/2}x}_2
       +\norm{A_s\Theta_s^{1/2}y}_2\right\}.                  \label{eq:derivative-difference}
\end{align}
Here $A_s$ acts block-diagonally on the stacked observations.  Since the
geodesic remains in $\mathcal C_R$,
$\norm{\Theta_s^{1/2}(x-y)}_2\leq e^{R/2}\norm{x-y}_2$.  The
upper Hessian bound gives
\begin{equation}
 {1\over M}\norm{A_s\Theta_s^{1/2}x}_2^2\leq L_Rd_*^2,
 \qquad
 {1\over M}\norm{A_s\Theta_s^{1/2}y}_2^2\leq L_Rd_*^2.        \label{eq:upper-hessian-data}
\end{equation}
Thus
\begin{equation}
 \left|{d\over ds}(f_x-f_y)(\Theta_s)\right|
 \leq {2e^{R/2}\sqrt{L_R}\over\sqrt M}\norm{x-y}_2d_*.        \label{eq:derivative-final}
\end{equation}
Integrating, substituting into \eqref{eq:kkt-sum}, and cancelling $d_*$
proves \eqref{eq:solution-lipschitz}; the case $d_*=0$ is immediate.
\end{proof}
\section{Equivariant operator-norm localization}
\label{sec:localization}
Write the canonical decomposition of a point in $\cP$ as
\begin{equation}
 \theta=e^{h_0}\bigotimes_{a=1}^kP_a,
 \qquad \det P_a=1,
 \qquad F_a(\theta)=\log P_a\in\Sym_0(d_a).                    \label{eq:canonical-log}
\end{equation}
We first relate factor logs to the full affine-invariant geometry.
\begin{lemma}[Factor-log Lipschitz bound]
\label{lem:factor-log}
For $\theta,\theta'\in\mathcal C_R$,
\begin{equation}
 \norm{F_a(\theta)-F_a(\theta')}_F
 \leq C_R\sqrt{d_a}\,d(\theta,\theta').                        \label{eq:factor-log-lipschitz}
\end{equation}
\end{lemma}
\begin{proof}
The canonical factor log is recovered from the full log by
\begin{equation}
 F_a(\theta)=\PiZero\left[{d_a\over D}
                \tr_{-a}\{\log\theta\}\right],                \label{eq:recover-factor-log}
\end{equation}
where $\PiZero$ is traceless projection.  The induced
Frobenius-to-Frobenius norm of $(d_a/D)\tr_{-a}$ equals
$\sqrt{d_a/D}$.
It remains to compare the full matrix logarithms.  Let $A_s$ be the
affine-invariant geodesic from $A$ to $B$, with velocity $V_s$.  The order
interval in \eqref{eq:thompson-ball} is geodesically convex, and along this
path
\begin{equation}
 \norm{D\log_{A_s}[V_s]}_F
 \leq e^R\norm{V_s}_F
 \leq e^{2R}\norm{A_s^{-1/2}V_sA_s^{-1/2}}_F.                 \label{eq:dlog-bound}
\end{equation}
The final norm is the constant affine-invariant speed.  Integration gives
\begin{equation}
 \norm{\log A-\log B}_F
 \leq e^{2R}\norm{\log(A^{-1/2}BA^{-1/2})}_F.                  \label{eq:log-distance}
\end{equation}
Combining \eqref{eq:recover-factor-log}, \eqref{eq:log-distance} and
$\norm{\log(\theta^{-1/2}\theta'\theta^{-1/2})}_F=\sqrt D\,
d(\theta,\theta')$ proves the lemma.
\end{proof}
For every $x\in\cG$, the constrained minimizer $\theta_x$ is well defined:
it exists by compactness of $\mathcal C_R$ and is unique by
Lemma~\ref{lem:hessian-transport}.
By Lemmas~\ref{lem:solution-map} and \ref{lem:factor-log}, the map
\begin{equation}
 x\longmapsto F_a(\theta_x),\qquad x\in\cG,                    \label{eq:good-factor-map}
\end{equation}
is $C_R\sqrt{d_a/M}$-Lipschitz into the Hilbert space
$\Sym_0(d_a)$.  Kirszbraun's extension theorem
\citep{kirszbraun1934transformationen} extends it to all of $\R^M$ without
increasing the Lipschitz constant.  Let $\widetilde F_a$ be one such
extension and average it over the local orthogonal group:
\begin{equation}
 \overline F_a(x)=
 \int_{\prod_bO(d_b)}U_a^T\widetilde F_a
 \left((\otimes_bU_b)x\right)U_a\,dU.                          \label{eq:reynolds-average}
\end{equation}
Here $dU$ denotes normalized Haar probability measure on the compact product
group.
The extension $\widetilde F_a$ is Lipschitz and hence continuous.  Since
the local orthogonal group is compact, the integrand in
\eqref{eq:reynolds-average} is continuous and bounded, so the displayed
Bochner integral is well defined and measurable.
The event $\cG$ and the constrained minimizer are equivariant.  Hence, for
$x\in\cG$, every integrand in \eqref{eq:reynolds-average} equals
$F_a(\theta_x)$, so $\overline F_a$ agrees with
\eqref{eq:good-factor-map} on $\cG$.  It remains
$C_R\sqrt{d_a/M}$-Lipschitz, takes values in $\Sym_0(d_a)$, and is globally
equivariant.
For standard Gaussian data, $\E\overline F_a(x)$ commutes with every
orthogonal $d_a\times d_a$ matrix.  It is therefore a scalar multiple of
the identity.  Since it is traceless, the expectation is exactly zero.
For every unit vector $u$, the scalar map
$x\mapsto u^T\overline F_a(x)u$ is centered and
$C_R\sqrt{d_a/M}$-Lipschitz.  Gaussian concentration
\citep[Chapter 5]{boucheron2013concentration} therefore implies
\begin{equation}
 \Prob\left\{|u^TF_a(\theta_x)u|>s,\ x\in\cG\right\}
 \leq2\exp\left\{-{cMs^2\over d_a}\right\}.                    \label{eq:scalar-concentration}
\end{equation}
No exceptional-set bias appears: concentration is applied to the global
averaged extension, which is exactly centered, and the extension equals
the optimizer map on $\cG$.
Let $\mathcal N_a$ be a $1/4$-net of the unit sphere with
$|\mathcal N_a|\leq9^{d_a}$.  For symmetric $A$,
$\norm A_{\rm op}\leq2\max_{u\in\mathcal N_a}|u^TAu|$.
Taking $s=R/(8k)$ in \eqref{eq:scalar-concentration} gives
\begin{equation}
 \Prob\left\{\norm{F_a(\theta_x)}_{\rm op}>{R\over4k},
                 \ x\in\cG\right\}
 \leq2\exp\left\{Cd_a-{cMR^2\over k^2d_a}\right\}.             \label{eq:operator-localization}
\end{equation}
Under \eqref{eq:main-threshold}, a union bound over the modes makes the
right-hand side at most $Ck\exp\{-ct^2\dmax\}$ after increasing the
universal sample constant.
The same net argument at the statistical scale will give the Thompson rates.
Taking $s=C_0td_a/\sqrt M$ in \eqref{eq:scalar-concentration}, with $C_0$
sufficiently large, yields
\begin{equation}
 \Prob\left\{\norm{F_a(\theta_x)}_{\rm op}
       >{Ctd_a\over\sqrt M},\ x\in\cG\right\}
 \leq 2e^{-ct^2d_a}.                                         \label{eq:adaptive-operator-localization}
\end{equation}
\section{Proof of the main theorem}
\label{sec:main-proof}
By equivariance, it suffices first to consider true precision $I_D$.  On
the gradient event \eqref{eq:gradient-concentration}, take
\begin{equation}
 \eps={t\dmax\over\sqrt M}.                                    \label{eq:epsilon-final}
\end{equation}
The sample condition ensures $\eps<1$, and
\eqref{eq:gradient-concentration} yields
\begin{equation}
 \norm{\nabla f_x(I_D)}
 \leq C\sqrt{k}\,{t\dmax\over\sqrt M}                          \label{eq:gradient-final}
\end{equation}
outside an event of probability at most $Ck e^{-ct^2\dmax}$.
Because $I_D\in\mathcal C_R$ and $\theta_x$ minimizes over
$\mathcal C_R$, strong convexity and comparison with $I_D$ imply
\begin{equation}
 d(I_D,\theta_x)
 \leq {2\over\lambda_R}\norm{\nabla f_x(I_D)}
 \leq C_R\sqrt{k}\,{t\dmax\over\sqrt M}.                       \label{eq:distance-to-identity}
\end{equation}
In the canonical decomposition \eqref{eq:canonical-log},
$|h_0|\leq d(I_D,\theta_x)$.  On the event
\eqref{eq:operator-localization} for every mode,
\begin{align}
 \norm{\log\theta_x}_{\rm op}
 &\leq |h_0|+\sum_a\norm{F_a(\theta_x)}_{\rm op} \notag\\
 &\leq C_R\sqrt{k}\,{t\dmax\over\sqrt M}+{R\over4}
 <{R\over2},                                                   \label{eq:strict-interior}
\end{align}
where the final inequality follows by increasing $C$ in
\eqref{eq:main-threshold}.  Thus the constrained minimizer lies strictly
inside $\mathcal C_R$ and is an unconstrained critical point of $f_x$.
Global geodesic convexity makes it a global minimizer.  It is unique: if a
second minimizer existed, the objective would be constant on the joining
geodesic, contradicting the strictly positive Hessian on the initial
segment contained in $\mathcal C_R$.
Equation \eqref{eq:distance-to-identity} and
\eqref{eq:metric-normalization} give
\begin{equation}
 d_{\rm FR}(\widehat\Theta,I_D)
 \leq C{\sqrt{k}\,\dmax\over\sqrt n}\,t.                       \label{eq:full-rate-proof}
\end{equation}
For the factors, write
$\widehat\Theta_a=e^{h_0/k}P_a$ and $F_a=\log P_a$.  The coordinate
normalization gives
\begin{equation}
 d(I_D,\widehat\Theta)^2
 =h_0^2+\sum_b{\norm{F_b}_F^2\over d_b},                       \label{eq:coordinate-distance}
\end{equation}
while
\begin{align}
 d_{\rm FR}(\widehat\Theta_a,I_{d_a})^2
 &={1\over2}\left({d_ah_0^2\over k^2}+\norm{F_a}_F^2\right)
 \leq {d_a\over2}d(I_D,\widehat\Theta)^2.                      \label{eq:factor-distance}
\end{align}
Combining \eqref{eq:factor-distance} and
\eqref{eq:distance-to-identity} proves \eqref{eq:main-factor-rate}.
For the Thompson bounds, \eqref{eq:adaptive-operator-localization} and
$|h_0|\leq C_R\sqrt{k}\,t\dmax/\sqrt M$ give
\begin{align}
 d_{\rm op}(\widehat\Theta_a,I_{d_a})
 &=\norm{{h_0\over k}I_{d_a}+F_a}_{\rm op} \notag\\
 &\leq {Ct\over\sqrt M}
       \left(d_a+{\dmax\over\sqrt{k}}\right).                \label{eq:factor-op-proof}
\end{align}
A union bound proves Corollary~\ref{cor:factor-op}; the gradient and
localization events are absorbed into the last term of
\eqref{eq:op-success-probability} because
$ke^{-ct^2\dmax}\leq\sum_ae^{-ct^2d_a}$.  If $d_a=\dmax$, the
failure probability in \eqref{eq:adaptive-operator-localization} is
$O(e^{-ct^2\dmax})$, which is absorbed into
\eqref{eq:success-probability}; this proves
\eqref{eq:main-largest-op-rate}.
For a general true product precision $\Theta$, whiten the observations by
$\Theta^{1/2}$.  To verify the factor normalization explicitly, write the
true precision and the MLE for the whitened data in canonical form as
$\Theta=\bigotimes_a\Theta_a$ and
$\widehat\Omega=\bigotimes_a\widehat\Omega_a$.  The determinant roots of
$\Theta_a^{1/2}\widehat\Omega_a\Theta_a^{1/2}$ are the products of the two
common determinant roots and hence are again identical across $a$.
Consequently these matrices are precisely the canonical factors of
$\Theta^{1/2}\widehat\Omega\Theta^{1/2}$.  The MLE transforms by this product
congruence, while both the Fisher--Rao and Thompson metrics are congruence
invariant, so all the same factor bounds hold at $\Theta$; see
\citet[Fact 2.6]{franks2026near}.  The failure probabilities are supplied by
Lemma~\ref{lem:full-lie}, \eqref{eq:gradient-concentration} and
\eqref{eq:operator-localization}, together with
\eqref{eq:adaptive-operator-localization} for the Thompson conclusions.
This proves Theorem~\ref{thm:main} and Corollary~\ref{cor:factor-op}.
\section{Information-theoretic comparison}
\label{sec:lower-bound}
For $r_*>0$, define
\[
 \mathcal B_d^0(r_*)=
 \{A\in\PD(d):\det A=1,\ d_{\rm op}(A,I_d)\leq r_*\}.
\]
\begin{proposition}[Gaussian submodel lower bounds]
\label{prop:lower-bound}
There are universal constants $c,r_*>0$ such that the following holds.  For
every estimator $\widetilde\Theta_a$ of the canonical factor in mode $a$,
\begin{align}
 \sup_{\Theta_a\in\mathcal B_{d_a}^0(r_*)}
 \Prob_{\Theta_a}\left\{
 d_{\rm FR}(\widetilde\Theta_a,\Theta_a)
 \geq c\min\left(\sqrt{d_a},{d_a^{3/2}\over\sqrt{nD}}\right)\right\}
 &\geq {1\over2},                                             \label{eq:lower-fr-factor}\\
 \sup_{\Theta_a\in\mathcal B_{d_a}^0(r_*)}
 \Prob_{\Theta_a}\left\{
 d_{\rm op}(\widetilde\Theta_a,\Theta_a)
 \geq c\min\left(1,{d_a\over\sqrt{nD}}\right)\right\}
 &\geq {1\over2}.                                             \label{eq:lower-op-factor}
\end{align}
Here all factors other than $a$ are fixed at the identity.  For every
$\PD(D)$-valued estimator $\widetilde\Theta$ of the full precision,
\begin{equation}
 \sup_{\Theta\in\cP}
 \Prob_\Theta\left\{
 d_{\rm FR}(\widetilde\Theta,\Theta)
 \geq c\min\left(\sqrt D,{\dmax\over\sqrt n}\right)\right\}
 \geq {1\over2}.                                              \label{eq:lower-fr-full}
\end{equation}
Consequently, the corresponding minimax expected squared risks are bounded
below by constant multiples of
$\min\{d_a,d_a^3/(nD)\}$, $\min\{1,d_a^2/(nD)\}$ and
$\min\{D,\dmax^2/n\}$, respectively.
\end{proposition}
\begin{proof}
We first prove the unit-determinant version of the classical Gaussian lower
bound.  Let $N$ observations have distribution $N(0,A^{-1})$ in dimension
$d$.  For symmetric traceless $H,K$ with operator norm at most one and
$0<\delta\leq r_*$, put $A_H=e^{\delta H}$.  If $r_*$ is a sufficiently
small numerical constant, then
\begin{align}
 d_{\rm FR}(A_H,A_K)&\geq c\delta\norm{H-K}_F,                 \label{eq:local-fr-lower}\\
 d_{\rm op}(A_H,A_K)&\geq c\delta\norm{H-K}_{\rm op}.         \label{eq:local-op-lower}
\end{align}
Indeed, $A_H,A_K$ lie in a fixed order interval.  Equation
\eqref{eq:log-distance}, applied in that interval, gives
\[
 \delta\norm{H-K}_F
 \leq C\norm{\log(A_H^{-1/2}A_KA_H^{-1/2})}_F,
\]
which proves \eqref{eq:local-fr-lower}.  For the operator norm, if
$r=d_{\rm op}(A_H,A_K)$, then
$e^{-r}A_H\preceq A_K\preceq e^rA_H$.  Hence
$\norm{A_H-A_K}_{\rm op}\leq Cr$, and the operator-Lipschitz bound for the
matrix logarithm on a fixed positive interval, obtained from its resolvent
integral representation, gives
$\delta\norm{H-K}_{\rm op}\leq Cr$, proving
\eqref{eq:local-op-lower}.  Moreover,
\begin{equation}
 D\{N(0,A_H^{-1})^{\otimes N}\|N(0,I_d)^{\otimes N}\}
 ={N\over2}\{\tr(e^{-\delta H})-d\}
 \leq CN\delta^2\norm H_F^2.                                 \label{eq:local-kl}
\end{equation}

For the Fisher--Rao bound, let $r=\lfloor d/2\rfloor$.  Volumetric packing of
the $r(d-r)$-dimensional Grassmann manifold gives $e^{cd^2}$ rank-$r$
orthogonal projections $P_j$ such that
$\norm{P_i-P_j}_F\geq c\sqrt d$ for $i\ne j$.  Put
$H_j=P_j-(r/d)I_d$.  Then $H_j\in\Sym_0(d)$,
$\norm{H_j}_{\rm op}\leq1$, $\norm{H_j}_F^2\leq d$, and the same pairwise
Frobenius separation holds.  Apply Fano's inequality to $A_{H_j}$ with
$\delta^2=c_0\min\{1,d/N\}$.  Equations
\eqref{eq:local-fr-lower} and \eqref{eq:local-kl} give
\begin{equation}
 \sup_{A\in\mathcal B_d^0(r_*)}
 \Prob_A\left\{d_{\rm FR}(\widetilde A,A)
 \geq c\min(\sqrt d,d/\sqrt N)\right\}\geq{1\over2}.       \label{eq:gaussian-fr-lower}
\end{equation}
For the Thompson bound, a volumetric packing of real projective space gives
$e^{cd}$ unit vectors $u_j$ with
$|u_i^Tu_j|\leq1/2$ for $i\ne j$, and put
$H_j=u_ju_j^T-d^{-1}I_d$.  These matrices are traceless,
$\norm{H_j}_F\leq1$, and
$\norm{H_i-H_j}_{\rm op}\geq\sqrt3/2$.  Fano's inequality with
$\delta^2=c_0\min\{1,d/N\}$, together with
\eqref{eq:local-op-lower} and \eqref{eq:local-kl}, gives
\begin{equation}
 \sup_{A\in\mathcal B_d^0(r_*)}
 \Prob_A\left\{d_{\rm op}(\widetilde A,A)
 \geq c\min(1,\sqrt{d/N})\right\}\geq{1\over2}.             \label{eq:gaussian-op-lower}
\end{equation}

In the tensor submodel with all nuisance factors fixed at the identity,
flattening over mode $a$ gives
\begin{equation}
 N_a={nD\over d_a}                                             \label{eq:effective-samples}
\end{equation}
independent $d_a$-dimensional Gaussian observations.  Since the truths in
$\mathcal B_{d_a}^0(r_*)$ have determinant one, their canonical factors are
exactly $\Theta_a$ in mode $a$ and identities elsewhere.  Substituting
$d=d_a$ and $N=N_a$ into \eqref{eq:gaussian-fr-lower} and
\eqref{eq:gaussian-op-lower} proves
\eqref{eq:lower-fr-factor}--\eqref{eq:lower-op-factor}.

It remains to prove \eqref{eq:lower-fr-full}.  Let
$\widetilde\Theta^{\cP}$ be the metric projection of
$\widetilde\Theta$ onto $\cP$.  The set $\cP$ is closed and geodesically
convex in the Hadamard manifold $\PD(D)$.  Geodesic convexity follows by
factorizing the affine-invariant geodesic.  For closedness, along a convergent
sequence in $\cP$ the condition numbers are bounded; after canonical
normalization the factors range over compact sets, and a subsequential factor
limit represents the matrix limit.  Hence the projection exists uniquely
and is measurable.  The nearest-point
property implies, for every $\Theta\in\cP$,
\begin{equation}
 d_{\rm FR}(\widetilde\Theta^{\cP},\Theta)
 \leq2d_{\rm FR}(\widetilde\Theta,\Theta).                    \label{eq:nearest-product-rounding}
\end{equation}
Choose a largest mode $a$ and restrict the truth to
\[
 \Theta(A)=I_{d_1}\otimes\cdots\otimes A\otimes\cdots\otimes I_{d_k},
 \qquad A\in\mathcal B_{\dmax}^0(r_*).
\]
Write
$\widetilde\Theta^{\cP}=e^h\bigotimes_bP_b$ with $\det P_b=1$, and set
$\widetilde A=e^hP_a$.  This is a measurable estimator because the canonical
decomposition is continuous, for example by \eqref{eq:recover-factor-log}.  If
\[
 L_a=\log(A^{-1/2}\widetilde A A^{-1/2}),
 \qquad L_b=\log P_b\quad(b\ne a),
\]
then the embedded mode matrices are Frobenius orthogonal and
\begin{align}
 2d_{\rm FR}(\widetilde\Theta^{\cP},\Theta(A))^2
 &=\norm{L_a^{(a)}+\sum_{b\ne a}L_b^{(b)}}_F^2 \notag\\
 &\geq {D\over\dmax}\norm{L_a}_F^2
 ={2D\over\dmax}d_{\rm FR}(\widetilde A,A)^2.                \label{eq:metric-embedding}
\end{align}
Apply \eqref{eq:gaussian-fr-lower} to $\widetilde A$ with
$d=\dmax$ and $N=nD/\dmax$.  Multiplication of the lower radius by
$\sqrt{D/\dmax}$ gives
$c\min\{\sqrt D,\dmax/\sqrt n\}$.  Equations
\eqref{eq:nearest-product-rounding} and \eqref{eq:metric-embedding} therefore
give \eqref{eq:lower-fr-full}.  The expected-risk conclusions follow from
$\E Z^2\geq r^2\Prob\{Z\geq r\}$.
\end{proof}

Under \eqref{eq:main-threshold}, the nonsaturated terms in all three lower
bounds apply.  For a largest factor, the Fisher--Rao radius in
\eqref{eq:main-factor-rate} differs from the lower radius in
\eqref{eq:lower-fr-factor} by at most $C\sqrt{k}t$.  The same comparison
holds for the full precision.  The Thompson bound
\eqref{eq:main-largest-op-rate} matches \eqref{eq:lower-op-factor} up to the
confidence parameter $t$.  Thus constant operator accuracy requires
$nD=\Omega(\dmax^2)$, and the dependence of the sample threshold in
Theorem~\ref{thm:main} on $\dmax$ is optimal for fixed $k$.
\section{Discussion}
Theorem~\ref{thm:main} resolves the explicit question posed in Section 6
of \citet{franks2026near}: the guarantees of their tensor-normal MLE theorem
continue to hold when $\dmax^3$ in the sample threshold is replaced by
$\dmax^2$.  The proof uses the full-matrix form of the random-channel bound,
exact Hessian transport and equivariant concentration of a constrained
optimizer.  We do not determine the optimal dependence on $k$ or prove
flip-flop convergence at the improved threshold.
\section*{Acknowledgement}
GPT-5.6 Sol and Claude Fable 5 (both used under the highest-tier subscription mode) were used to assist with proof development, verification and paper writing. 
\bibliographystyle{biometrika}
\bibliography{references}
\renewcommand{\printhistory}{}
\end{document}